\documentclass[]{article}
\usepackage{amssymb,amsmath,amsthm,mathrsfs}

\newcommand{\Z}{{\mathbf Z}}
\newcommand{\R}{\mathbf{R}}
\newcommand{\T}{\mathbf{T}}

\renewcommand{\P}{\mathrm{P}}
\newcommand{\Dom}{\mathrm{Dom}}
\renewcommand{\L}{\mathrm{L}}
\newcommand{\D}{\mathrm{Dom}}
\renewcommand{\L}{\mathrm{L}}
\newcommand {\E}{\mathrm{E}}
\newcommand{\1}{{\bf 1}}
\newcommand{\intT}{\int_0^{2\pi}}
\newcommand{\intt}{\int_0^t}
\renewcommand{\d}{\rho}
\renewcommand{\Re}{\text{\rm Re}\,}

\newcommand{\sL}{\mathcal{L}}
\renewcommand{\-}{\hspace{.01 in}}

\newcommand{\sB}{\mathcal{B}}
\newcommand{\sF}{\mathcal{F}}

\newcommand{\cS}{\mathscr{S}}

\renewcommand{\phi}{\varphi}

\newcommand{\lip}{\mathrm{Lip}}
\newcommand{\ve}{\epsilon}
\renewcommand{\ss}{\hat{s}}

\newcommand{\yy}{\hat{y}}

\newcommand{\al}{\hat\alpha}

\newcommand{\G}{\Gamma}
\newtheorem{stat}{Statement}[section]
\newtheorem{hp}{H}
\newtheorem{proposition}[stat]{Proposition}

\newtheorem{lma}[stat]{Lemma}
\newtheorem{crl}[stat]{Corollary}
\newtheorem{rmk}[stat]{Remark}
\newtheorem{prp}[stat]{Proposition}
\theoremstyle{definition} 
\newtheorem{thm}[stat]{Theorem}
\newtheorem{definition}[stat]{Definition}
\newtheorem{remark}[stat]{Remark}

\numberwithin{equation}{section}
\begin{document}


\title{\bf Regularity of the density for a stochastic heat equation}	
\author{Pejman Mahboubi}

%
\maketitle
\begin{abstract}
  We study the smoothness of the density of the solution to the nonlinear heat equation $\partial_t u=\sL u+\sigma(u)\dot w$ on a torus with a periodic boundary condition, where $\sL$ is the generator of a L\'evy process on the torus. We use Malliavin calculus techniques to show that the law of the solution has a density with respect to the Lebesgue measure for all $t>0$ and $x\in\R$. 		
		
		
\end{abstract}
{\bf{Acknowledgements.}} The author would like to thank his advisor, Professor D. Khoshnevisan, for many fruitful and encouraging discussions. His  suggestions, corrections and valuable comments made this work possible.

\section{Introduction and the main result}
Let $\{\dot w(t,x)\}_{t\geq0,x\in[0,2\pi]}$ denote space-time white noise on the torus $\T$, and let $\sigma:\R\to\R$ be a nice function. For every $T\in[0,\infty]$, we define $E_T:=[0,T]\times\T$, and let $E$ denote $E_\infty=\bigcup_{T>0}\E_T$. We aim  to establish a sufficient condition that ensures that the solution to the parabolic stochastic partial differential equation [SPDE]
\begin{equation}\label{A SPDE}
\left|\begin{array}{ll}
\partial_t u(t,x)=\sL u(t,x)+\sigma(u(t,x))\dot w&(t,x)\in E,\\
u(t,0)=u(t,2\pi)&t\geq 0,\\
u(0,x)=u_0(x)&x\in\T,
\end{array}\right.
\end{equation}
has a density which is smooth at all $t>0$ and $x\in\T$.  Here $\sL$ is the $L^2(\T)$-generator of a L\'evy process $X:=\{X_t\}_{t\geq 0}$, and acts only on the variable $x$, and $u_0$ is a bounded, measurable real function on $\mathbf T$. We denote by $C_b^\infty(\R)$ the space of all smooth functions on $\R$ which are bounded together with all their derivatives. Let $\varPhi:\mathbf{Z\to C}$ denote the characteristic exponent of $X$ normalized so that $\E\exp(inX_t)=\exp(-t\varPhi(n))$ for all $n\in\Z$ and $t>0$. In other words, $\varPhi$ is the Fourier multiplier of $\sL$ and $\hat\sL(n)=-\varPhi(-n)$ holds for all $n\in\Z$. The central result of this work is the the following. 
\begin{thm}\label{main}
Let $u$ be the mild solution to the equation~\eqref{A SPDE}, where $\sigma\in C_b^\infty(\R)$ and suppose that there is a $\kappa>0$ such that $\sigma\geq\kappa>0$. Assume that there exist finite constants $c, C\geq0$  and $1<\alpha\leq\beta\leq2$, such that 
\begin{equation}\label{hype1}
c|n|^\alpha\leq\Re\varPhi(n)\leq C|n|^\beta,
\end{equation}
for all $n\geq1$. If $\alpha\geq2\beta/(\beta+1)$, then $u(t,x)$ has a smooth density at every $t>0$ and $x\in\T$. This holds, in particular when 
$$c\; n^{\frac{4}{3}+\epsilon}\leq\varPhi(n)\hspace{.2 in}\forall n\geq1,$$
where $0<\epsilon<\frac{2}{3}$.
\end{thm}

In Hypothesis  {\bf{H}\-\ref{condition0}} below, we will discuss briefly how the existence of the mild solution imposes a restriction on the underlying L\'evy process $X$. We would like to remark that when ~\eqref{A SPDE} is linear, if $\alpha,\beta\leq1$, then a solution does $not$ exist. This observation might explain why we need the condition $\alpha>1$ in Theorem~\ref{main}.

A significant byproduct of our method is the existence of finite Liapounov exponents for the Malliavin derivatives. Indeed, the existence of Malliavin derivatives is proven to be connected closely to the growth of $t\to Dv_n(t,x)$, where $v_n(t,x)$ is the approximating function in Picard scheme. The choice of our family of seminorms allows us to interpret this growth as a result about the Liaponov exponents of the the Malliavin derivatives. More specifically we have the following result :
\begin{thm}
Let $u=u(t,x)$ be the solution of~\eqref{A SPDE}. Under condition \mbox{{\bf{H}\-\ref{condition0}}} below, $u\in \mathbf D^{m,p}$ for all $m\geq1$ and $p\geq1$, and 
\begin{equation}
\limsup_{t\to\infty}\frac{\ln\E\|D^mu(t,x)\|_{H^{\otimes m}}^p}{t}<\infty.
\end{equation}
\end{thm}

Malliavin's calculus is an appropriate tool for the study of the existence and regularity of densities of functionals on the Wiener space. This method is normally implemented in two steps: 
\begin{description}
\item[Step 1] is to prove the existence of the Malliavin derivatives of all order,  and 
\item[Step 2] is the study of the corresponding Malliavin matrix and existence of the negative moments.
\end{description}

In ``Step 1'' we offer a new method, which, in contrast to the other works, does not rely on the approximations that use the properties of the transition probabilities of the L\'evy process.  This feature of our proof has enabled us to prove the Malliavin differentiability  of the solution for all L\'evy processes, for which the existence of the mild solution is proved. To emphasize, we mention that the conditions stated in terms of $\alpha$ and $\beta$ are not used in this step.

In ``Step 2'' we followed carefully \cite[pages 97-98]{minicourse}, and could find an ``$\epsilon$-room'' to extend the results from Brownian motion to a large group of L\'evy processes, characterized by the rate of the growth of their L\'evy exponents.

The idea for the existence and uniqueness of the solutions to ~\eqref{A SPDE} come from \cite{FKN} and \cite{FK}. A linearized version of \eqref{A SPDE} on $\R$, with vanishing initial data, in which the noise is additive, i.e, 
\begin{eqnarray}\label{linearized}
\left|\begin{array}{l}
\partial_tu(t,x)=\sL u(t,x)+\dot w,\\
u(0,x)=0,
\end{array}
\right.
\end{eqnarray}
is studied by Foondun et al. in  \cite{FKN}. They have shown a one-to-one correspondence between the existence of a unique random field solution to~\eqref{linearized} and  the existence of the a local time for the symmetrized underlying L\'evy process $\bar X$, where 
\begin{equation}\label{sym}
\bar X_t=X_t-X'_t\;\;\;\forall t\geq0,
\end{equation}
and $X'=\{X_t\}_{t\geq0}$ is an independent copy of $X$.

In \cite{FK}, the authors consider a multiplicative white noise and study the existence and uniqueness of the mild solution to the equation 
\begin{equation}\label{R}
\left|\begin{array}{ll}
\partial_t u=\sL u+\sigma(u)\dot w&t\geq0,x\in\R\\
u(x,0)=u_0(x)&x\in\R
\end{array}\right.
\end{equation}
with a nonnegative initial data $u_0$. In this paper, Foondun and Khoshnevisan combine the existence result of \cite{FKN} with a theorem of Hawkes \cite{Hawkes} to show that ~\eqref{R} has a strong solution if and only if $\upsilon(\beta)<\infty$, for some $\beta>0$ where 
\begin{equation*}
\upsilon(\beta):=\frac{1}{2\pi}\int_{-\infty}^\infty\frac{d\xi}{\beta+2\Re\phi(\xi)},
\end{equation*}
where $\phi$ denotes the characteristic exponent of $X$.
Therefore, it is natural to consider a solution to equation~\eqref{A SPDE} under a similar condition. We define
\begin{equation*}
\Upsilon(\beta):=\frac{1}{4\pi^2}\sum_{-\infty}^\infty \frac{1}{\beta+2\Re\varPhi(n)}.
\end{equation*}
It is clear from the definition of $\varPhi$ in ~\eqref{phi} below, that $\Upsilon(\beta)<\infty$ for some $\beta$, when Hypothesis {\bf{H}\-\ref{condition0}} below holds. The function $\Upsilon$ continues to have a crucial role in ``Step 1.'' The convergence of all Picard iterations relies on the fact that $\Upsilon(\beta)\to0$ as $\beta\to\infty$. 

The rest of of this paper is organized as follows. In section 2 we collect some results about L\'evy processes that are relevant to our results. At the end of this section we also state the result for existence and uniqueness of the solution to equation~\eqref{A SPDE}. Our proof for this result follows closely the proof in \cite{FK} and we avoid to repeat it here. In section 3 we review elements of the Malliavin calculus as economically as possible. In section 4 we show that the Malliavin derivative of $u$ of all order exists; here $u=u(t,x)$ is the solution to~\eqref{A SPDE}. In Section 5 we analyze the Malliavin matrix.

\section{L\'evy processes on the torus}
In this section we review some results about L\'evy processes.  This material will be used in sequel. 
Let $\{Y_t\}$ be a L\'evy process on $\R$, with $\phi$ as its characteristic exponent; i.e,
\begin{equation}
\E e^{i\lambda(Y_{s+t}-Y_s)}=e^{-t\phi(\lambda)},\hspace{.3 in}s, t\geq 0,\lambda\in\R.
\end{equation}	
As we discussed before --- see the paragraph before~\eqref{sym}--- we assume the following:
\begin{hp}\label{condition0}
Let $\bar{Y_t}=Y_t-Y'_t$, where $Y'$ is an independent copy f $Y$. $\bar{Y}_t$ has a local time. Results from \cite{Hawkes} imply that $\Upsilon(\beta)<\infty$, for all $\beta>0$.
\end{hp}
Then Lemma 8.1 in~\cite{FKN} tells us that under hypothesis {\bf{H}\-\ref{condition0}} process $Y_t$ has transition density $\{p_t(x,y)\}$ such that $\int_{\T}p_t(x,y)^2\;dy<\infty$ for all $x\in\T$. 

Let $\mathbf{T}:=[0,2\pi)$. Define a process $X_t$ on $\T$, via $Y_t$, by 
\begin{equation}\label{LPonTorus}
X_t=Y_t-2n\pi\hspace{.2 in}\text{when}\hspace{.2 in}2n\pi\leq Y_t<2(n+1)\pi.
\end{equation}

Let $\{q_t(x,\cdot)\}_{x\in T}$ denote the transition probability density for the process $X$.  A simple calculation shows that the transition densities of $X$ are given by 
\begin{equation}\label{q}
q_t(x,y)=\sum_{n=-\infty}^\infty p_t(x,y+2n\pi)\;\;\;\forall x,y\in\T.
\end{equation}
 Let us introduce a function $\varPhi:\Z\to\mathbf{C}$ by
\begin{equation}\label{phi}
\varPhi(n)=\phi(n)\hspace{.2 in}n\in\Z
\end{equation}
As it is shown below, $\varPhi$ is the characteristic exponent of the process $X$.
For $g\in L^2(\T)$ we have 
\begin{equation*}
g(x)=\sum_{n=-\infty}^\infty\hat{g}(n) e^{-inx},
\end{equation*}
where 
\begin{equation*}
\hat{g}(n)=\frac{1}{2\pi}\int_{\T} e^{inx}g(x)dx.
\end{equation*}
Since $q_t(x,y)$ is a function of $y-x$ for each $t\geq0$, we occasionally abuse notation and write $q_t(y-x)$ instead of $q_t(x,y)$.
\begin{lma}\label{dnstyPrpty}
Under Hypothesis {\bf H\-\ref{condition0}}, $q_t(x,\cdot)\in L^2(\T)$ for all $x\in\T$ and $t>0$. Furthermore,
\begin{eqnarray}\label{L2norm}
\hat{q}_t(x,n)=\frac{1}{2\pi}e^{inx}e^{-t\varPhi(n)},\hspace{.3 in}\|q_t(\cdot)\|^2_{L^2(\T)}=\frac{1}{4\pi^2}\sum_{n=-\infty}^\infty e^{-2t\Re\varPhi(n)},
\end{eqnarray}
and 
\begin{equation}\label{qfourier}
q_t(x,y)=\frac{1}{2\pi}\sum_{n=-\infty}^\infty e^{inx}e^{-t\varPhi(n)}e^{-iny}
\end{equation}
\end{lma}

\begin{proof}
To show that $q_t(x,\cdot)\in L^2([0,2\pi))$ we need only to show that its Fourier coefficients are in $\ell^2(\Z)$. 
We can write $\hat q_t$ in terms of $\varPhi$ as follows:
\begin{eqnarray}\label{char}
\hat{q}_t(x,n)=\frac{1}{2\pi}\int_{\T} e^{iny}\;q_t(y-x)dy
=\frac{1}{2\pi}e^{inx}\int_{-\infty}^\infty e^{inz}p_t(z)\;dz
=\frac{1}{2\pi}e^{inx}e^{-t\varPhi(n)}.
\end{eqnarray}
Therefore, $\varPhi$ is the characteristic exponent of the L\'evy process $X_t$. To prove the second formula, we need only to show that the sum in ~\eqref{L2norm} converges, because then this equation would be the Parseval identity. An application of Fubini and {\bf H\-\ref{condition0}} imply that 
\begin{equation}\label{parseval}
\int_0^\infty\sum_{n=-\infty}^\infty e^{-(\beta t+2t\Re\varPhi(n))}dt=4\pi^2\Upsilon(\beta)<\infty
\end{equation}
Therefore, by then the continuity of the integrand,
$$\sum_{n=-\infty}^\infty e^{-(\beta t+2t\Re\varPhi(n))}<\infty\hspace{.5 in}\forall t>0.$$
Therefore $\sum_{n=-\infty}^\infty e^{-2t\Re\varPhi(n)}<\infty$ for all $t>0$. Finally,~\eqref{qfourier} is a consequence of the inversion formula.
\end{proof}

The transition densities $q_t$ induces a semigroup $T_t$ on $L^2(\T)$ defined by
\begin{equation}\label{semop}
T_tf(x)=\E^xf(X_t)=\int_{\T} f(y)q_t(x,y)dy.
\end{equation}
\begin{lma}
The semigroup operator defined in ~\eqref{semop} is a convolution operator and 
\begin{eqnarray}\label{semigroup}
T_tf(x)=\sum_{n=-\infty}^\infty e^{-inx}e^{-t\varPhi(-n)}\hat{f}(n).
\end{eqnarray}
\end{lma}
\begin{proof}
Since
\begin{eqnarray*}
T_tf(x)
=\int_{\T} f(y)\frac{1}{2\pi}\left(\sum_{n=-\infty}^\infty e^{inx}e^{-t\varPhi(n)}e^{-iny}\right)dy,
\end{eqnarray*}
then an application of Fubini gives us the result.
\end{proof}

Let $\sL$ be the generator of $X_t$ in $L^2$ sense. This means
\begin{equation*}
\sL f(x)=\lim_{t\to0^+}\frac{T_tf(x)-f(x)}{t}\;\;\text{in}\; L^2(\T),
\end{equation*}
whenever the limit exists. It is natural to define 
\begin{eqnarray*}
\D[\sL]:=\left\{\phi\in L^2(\T):\sL(\phi):=\lim_{t\to0^+}\frac{T_t\phi-\phi}{t}\;\text{exists in}\;L^2(\T)\right\}.
\end{eqnarray*}
Next, we characterize $\D[\sL]$ in terms of the characteristic exponent.
\begin{prp} We have 
\begin{equation*}
\D[\sL]=\left\{f\in L^2(T):\sum_{n=-\infty}^\infty|\varPhi(-n)|^2|\hat f(n)|^2<\infty\right\}.
\end{equation*}
\end{prp}
\begin{proof}
From the definition and the continuity of the Fourier transform,
\begin{eqnarray*}
\widehat{\sL f}(n)=\lim_{t\to0^+}\frac{\widehat{T_tf}(n)-\hat f(n)}{t}=
\widehat{\sL f}(n)=\hat f(n)\lim_{t\to0^+}\frac{e^{-t\varPhi(-n)}-1}{t}=-\varPhi(-n)\hat f(n)
\end{eqnarray*}
Then $\varPhi(-n)\hat f(n)\in\ell^2(\Z)$, which equivalent to what we wanted to prove.
\end{proof}

Therefore $\sL$ can be viewed as a convolution operator with Fourier multiplier $\hat{\sL}(n)=-\varPhi(-n)$. We state the result as follows.

\begin{lma}
The $L^2(\T)$ generator $\sL$ of $T_t$ can be written as
\begin{equation}
\sL u_0(x)=-\sum_{n=-\infty}^\infty e^{inx}\varPhi(-n)\hat{u}_0(n)\hspace{.2 in} x\in[0,2\pi),
\end{equation}
for all $u_0\in L^2(\T)$.
\end{lma}

We borrow the following lemma from \cite{FK}; it plays a key rule in the proof of the existence of Malliavin derivatives.
\begin{lma}\label{supineq}
For all $\beta>0$,
\begin{equation}
\sup_{t>0}e^{-\beta t}\int_0^t\|q_s\|^2_{L^2(T)}ds\leq\int_0^\infty e^{-\beta s}\|q_s\|_{L^2(T)}^2ds=\Upsilon(\beta).
\end{equation}
\end{lma}
\begin{proof} The inequality is trivial. The equality follows from ~\eqref{L2norm} and ~\eqref{char},
\begin{align*}
 \int_0^\infty e^{-\beta s}\|q_s\|^2_{L^2(T)}ds&= \int_0^\infty e^{-\beta s}\sum_{n=1}^\infty|\hat{q}_s(x,n)|^2ds\\
&= \sum_{n=-\infty}^\infty\int_0^\infty e^{-\beta s}|\hat{q}_s(x,n)|^2ds=\Upsilon(\beta).
\end{align*}
This finishes the proof.
\end{proof}

The following results are used in ``Step 2'' of our proof, i.e the existence of the negative moments.

\begin{lma}
Let $1<\alpha\leq 2$. There is $C\in(0,\infty)$ such that 
\begin{equation}\label{est.1}
\lim_{\lambda\to0}\lambda^\frac{1}{\alpha}\sum_{n=1}^\infty e^{-n^\alpha\lambda}=C.
\end{equation}
\end{lma}
\begin{proof} The proof of this lemma for $\alpha=2$ is given in \cite[pages 34-35]{minicourse}, and some minor changes give the proof  for $\alpha<2$.
\end{proof}

Next we introduce the second hypothesis of the paper.
\begin{hp}\label{condition2}
Assume there are $1<\alpha<\beta\leq2$ and $0<C_1<C_2$ such that 
\begin{equation}\label{hypothesis2}
C_1|n|^{\alpha}\leq\Re\varPhi(n)\leq C_2|n|^{\beta}.
\end{equation}
\end{hp}

\begin{crl}
Let $\varPhi(n)$ denotes the L\'evy exponent of a L\'evy process with transition probability $q=q_t(x)$. If $\varPhi$ satisfies Hypothesis {\bf{H\-\ref{condition2}}}, then, for $t\in[0,T]$, there are constants $0<A_1<A_2$ depending on $T$, such that 
\begin{enumerate}
\item For all $t>0$,
\begin{equation}\label{qULB}
A_1t^{-\frac{1}{\beta}}\leq\|q_t\|_{L^2(\T)}^2\leq A_2t^{-\frac{1}{\alpha}};
\end{equation}
\item For any $\delta\in(0,T)$,
\begin{equation}\label{LB}
A_1\delta^{1-\frac{1}{\beta}}\leq\int_0^\delta\|q_t\|_{L^2(\T)}^2dt\leq A_2\delta^{1-\frac{1}{\alpha}}.
\end{equation}
\end{enumerate}
\end{crl}
\begin{proof}
We prove only the first part; the second part is obtained by integration. ~\eqref{hypothesis2} and~\eqref{L2norm} imply that 
\begin{eqnarray}\label{sandwich1}
\sum_{n=-\infty}^\infty e^{-2tC_2|n|^{\beta}}\leq 4\pi^2\|q_t\|_{L^2(\T)}^2\leq \sum_{n=-\infty}^\infty e^{-2tC_1|n|^{\alpha}}.
\end{eqnarray}
The first inequality in~\eqref{sandwich1} implies that 
\begin{equation}\label{t^beta}
t^{1/\beta}\|q_t\|_{L^2(\T)}^2\geq\frac{1}{(2C_2)^{1/\beta}4\pi^2} (2tC_2)^{1/\beta}\sum_{n=-\infty}^\infty e^{-2tC_2|n|^{\beta}}.
\end{equation}
By~\eqref{est.1} the right-hand-side of~\eqref{t^beta} converges to a number $B_1>0$. 
Therefore there is an $\epsilon_1>0$ such that 
\begin{equation*}
t^{1/\beta}\|q_t\|_{L^2(\T)}^2\geq B_1/2\hspace{.5 in}\forall t\in(0,\epsilon_1).
\end{equation*}
To extend the inequality to $t\in(0,T)$, we note that by ~\eqref{est.1}, $q_t\neq0$ and is continuous for $t>0$. Therefore, there is $A_1>0$ such that 
\begin{equation}\label{A_1}
t^{1/\beta}\|q_t\|_{L^2(\T)}^2\geq A_1 \hspace{.5 in}\forall t\in(0,T).
\end{equation}
Similarly, the second inequality in~\eqref{sandwich1} implies that there is $A_2$ such that  
\begin{equation}\label{A_2}
t^{1/\beta}\|q_t\|_{L^2(\T)}^2\leq A_2\hspace{.5 in}\forall t\in(0,T).
\end{equation}
~\eqref{A_1} and ~\eqref{A_2} imply ~\eqref{qULB}.
\end{proof}
\subsection{The mild solution, existence and uniqueness}
Equation~\eqref{A SPDE} is formal; we interpret it, in the Walsh sense, as the solution to the integral equation 
\begin{equation}\label{intequ}
u(t,x)=v(t,x)+\intT\int_0^t\sigma(u(s,y))q_{t-s}(y-x)\-w(ds\-dy),
\end{equation}
where the integral on the right is with respect to white noise, which also defines the a filtration $\{\sF_t\}_{t\geq0}$ via
$$\sF_t=\sigma\left(\dot w([0,s]\times A),\;0\leq s\leq t, \text{and}\;A\in\sB(\T),\; \lambda(A)<\infty\right),$$
where $\lambda$ is the Lebesgue measure on $\T$, normalized to have mass $1$. 
In~\eqref{intequ} a solution $u$ satisfying ~\eqref{intequ} is called the mild solution to ~\eqref{A SPDE}, if 
$$\sup_{(t,x)\in[0,T]\times\T}\E(|u(t,x)|^2)<\infty\hspace{.4 in}\text{for all}\;T<\infty.$$
We assume $\sF_t$ satisfies the usual condition for all $t\geq0$.
To introduce and analyze the properties of the solution, we introduce the following family of semi-norms: For any $T\in[0,\infty]$ and integers $p\geq 2$ define
\begin{equation}\label{norm}
\|f\|_{\beta,p, T}:=\left\{\sup_{0\leq t\leq T}\sup_{0\leq x\leq 2\pi}e^{-\beta t}\E(|f(t,x)|^p)\right\}^{1/p}.
\end{equation}
When $T=\infty$, we write $\|f\|_{\beta,p}$ instead of $\|f\|_{\beta,p,\infty}$.

\begin{rmk}
In this paper there are two families of semi-norms, indexed by two parameters. One is defined in ~\eqref{norm}, when $T=\infty$, which always comes with parameter $\beta$, and the other one is the norm  $\|\cdot\|_{k,p}$, on space $\mathbf{D}^{k,p}$, which is indexed by integers such as $k, m, n$ etc., and $p$.  For the sake of clarity, in sequel, when these two norms are both used, we will use a new notation defined by ~\eqref{Gamtx}.
\end{rmk}

\begin{thm}\label{existence}
Under the hypothesis {\bf{H\-\ref{condition0}}},~\eqref{A SPDE} has a solution $u$ that is unique up to a modification. The solution is bounded in $\|\cdot\|_{p,\beta}$ norm, for some some $\beta>0$, and all $p\geq 2$. Furthermore, when $u_0$ is continuous, $u$ is continuous in $L^p(\P)$ for all $p>0$.
\end{thm}
\begin{proof}
It is easy to check that, if we substitute $t=0$, then $u(0,x)=u(x)$, and also, if the solution exists, then $u(t,2\pi)=u(t,0)$.
Notice that if $v$ is defined by
\begin{equation}
v(t,x)=T_tu_0(x)\hspace{.2 in}(t,x)\in E,
\end{equation}
then $v$ satisfies $\partial_t v(t,x)=\sL v(t,x)$ weakly. Furthermore, the periodic condition $v(t,0)=v(t,2\pi)$ on $\T$ and the initial condition $v(0,x)=u_0(x)$ are satisfied.
That is $v$ is the Green's function for the operator $\partial_t-\sL$.  We consider the following Picard iteration. Define $v_0(t,x)=T_tu_0(x),$ and for $n\geq1$ set
\begin{equation}\label{picard}
v_{n+1}(t,x)=v_0(t,x)+\int_0^t\intT q_{t-r}(x,z)\sigma(v_n(r,z))\-w(dr dz).
\end{equation}
Then the existence of of the solution boils down to the convergence of $v_n$. We know for example, that $\lim_{n\to\infty}\|v_n-u\|_{\beta,p}=0$. For more details check Theorems 2.1 and A.1 in \cite{FK}.
\end{proof}

\section{Elements of Malliavin's calculus}
Most of this section is borrowed from \cite{Nualart} and \cite{Sanz-Sole}. Let  $C_p^\infty(\R^n)$ denote the space of the smooth real-valued functions $f$ on $\R^n$, such that $f$ and all its partial derivatives have at most polynomial growth. 

Let $\cS$ denote the class of smooth random variables such that a random variable $F\in\cS$ has the form
$F=f(w(h_1),\cdots,w(h_n),$
where $f\in C_p^\infty(\R^b)$, $h_i\in H:=L^2(E_T)$, and $w(h_i)$ denotes the Wiener integral
$$w(h_i)=\int_0^T\intT h_i(t,x)\-w(dx\;dt),\;\;1\leq i\leq n.$$
A stochastic process $u=\{u(t,x), t\geq 0, 0\leq x\leq 2\pi\}$ is called $adapted$ if $u(t,x)$ is $\sF_t$ measurable for any $(t\geq 0,x)\in E_T$.

Fix a (finite or infinite) time interval $[0,t]$, and denote by $L_a^2([0,t]\times\Omega)$ the set of square integrable and adapted processes.
The Skorohod integral is an extension of the It\^o integral in the following sense:
\begin{proposition}
 $L^2_a\subset\mathrm{Dom} \delta$, and the operator $\delta$ resricted to $L^2_a$ coincides with the It\^o integral, that is \begin{equation}
\delta(u)=\int_0^t\intT u(t,x)\-w(dx\;dt).
\end{equation}
\end{proposition}
\begin{prp}\label{iff}
Suppose $F\in L^2(\Omega)$, and let $J_n$ denote the projection to the $n$th Wiener chaos. $F\in \mathbf D^{k,2}$ if and only if 
\begin{equation}
\sum_{n=1}^\infty n^k\|J_nF\|_{L^2(\Omega)}^2<\infty.
\end{equation}
\end{prp}
To apply the Malliavin calculus to our problem, we usually need to compute the Malliavin derivatives of the integrals. In this regard the following proposition is useful. 
\begin{proposition}
If $u\in \mathbf D^{1,2}$ and $D_{t,x}u\in\Dom\-\delta$, then
\begin{equation}
D_{t,x}\delta(u)=\delta(D_{t,x}u)+u(t,x).
\end{equation}
\end{proposition}
\begin{proposition}\label{bddD}
Let $\{F_n,\;n\geq 1\}$ be a sequence of random variable converging to $F$ in $L^p(\Omega)$ for some $p>1$. Suppose that for some $k\geq 1$,
\begin{equation}
\sup_{n\geq 1}\|F_n\|_{k,p}<\infty.
\end{equation}
Then $F$ belongs to $\mathbf D^{k,p}$, and the sequence of derivatives $\{D^kF_n,\;n\geq 1\}$ converges to $D^kF$ in the weak topology of $L^p(\Omega; H)$.
\end{proposition}
\begin{remark}\label{L^1,2}
The space $\mathbf D^{1,2}(L^2(T))$, denoted by $\L^{1,2}$, coincides with the class of processes $u\in L^2(T\times\Omega)$ such that $u(t)\in \mathbf D^{1,2}$ for almost all $t\in T$, and there exists a measurable version of the two-parameter process $D_su_t$ verifying $\E\int_T\int_T (D_su_t)^2\mu(ds)\mu(dt)<\infty$. This space is included in $\Dom\-\delta$.
\end{remark}
\begin{proposition}\label{derivative}
Suppose that $u\in \L^{1,2}$. Furthermore assume that the following two conditions are satisfied:
\begin{enumerate}
\item For almost all $(s,y)\in E$ the process $\{D_{s,y}u(r,z), (r,z)\in E \}$ is Skorohod integrable;
\item There is a version of the process 
$$\left\{\intT\int_0^TD_{s,y}u(r,z)\-W(dr,dz), (s,y)\in E\right\},$$ 
which is in $L^2(\Omega\times E)$.
\end{enumerate}
Then $\delta(u)\in \mathbf D^{1,2}$ and we have
\begin{equation}\label{Dsigma}
D_{s,y}(\delta(u))=u(s,y)+\intT\int_0^TD_{s,y}u(r,z)\-W(dr,dz).
\end{equation}
\end{proposition}
If $F=(F^1,\cdots,F^n)$ is a random vector with $F^i\in\mathbf D^{1,1}$, then we define its Malliavin matrix $\gamma_F$ to be 
$$\gamma_F=(\langle F^i,F^j\rangle)_{1\leq i,j\leq m}.$$
Before we state the main theorem we introduce the following.
\begin{definition}
We say that a random vector $F=(F^1,\cdots,F^n)$ is nondegenerate if it satisfies the following conditions:
\begin{enumerate}
\item $F^i\in \mathbf D^\infty$ for all $i=1,\cdots,m;$
\item The matrix $\gamma_F$ satisfies $\E[det\gamma_F)^{-p}]<\infty$ for all $p\geq 2$.
\end{enumerate}
\end{definition}
The following is a key result.
\begin{thm}
If $F=(F^1,\cdots,F^n)$ is a nondegenerate random vector, then the law of $F$ possesses an infinitely-differentiable density. 
\end{thm}
While Proposition~\ref{derivative} allows us to prove an integral is in $\mathbf D^{1,2}$ it falls short of telling us whether or not it belongs to $\mathbf D^{1,p}$ for  $p>2$. The following proposition, which is a result of Meyer's inequality (see \cite{Nualart} page 72) states the required conditions for going from $p=2$ to $p>2.$
\begin{prp}\label{22p}
Let $F$ be a random variable in $\mathbf D^{k,\alpha}$, where $\alpha>1$. suppose that $D^iF$ belongs to $L^p(\Omega,H^{\otimes i})$ for $i=0,1,\cdots,k$, and for some $p>\alpha$. Then $F\in \mathbf D^{k,p}$.
\end{prp}
\section{The Malliavin derivatives of the solution}

In this section we assume that $\sigma\in C^\infty_b(\R)$ and the underlying L\'evy process satisfies {\bf{H}\-\ref{condition0}}. 
\begin{rmk}
We occasionally use the symbol $``\lesssim"$ in our proofs. By $X\lesssim Y$ we mean there is a positive $C$ such that $|X|\leq CY$. We might also subscript this by a parameter to denote dependence on this parameter. 
\end{rmk}
The main result in this section is the following.
\begin{thm}\label{4.1}
If $u$ is the solution to the Eq.~\eqref{intequ} then  $u(t,x)\in\mathbf D^\infty$ for almost all $(t,x)\in[0,T]\times[0,2\pi]$. Furthermore
\begin{equation}\label{Du}
D^k_\alpha u(t,x)=q_{t-\ss}(x,\yy)D^{k-1}_{\al}u(\ss,\yy)+\int_0^t\int_0^{2\pi}q_{t-r}(x,z)D^k_\alpha\sigma(u(r,z))\-w(dr,dz),
\end{equation}
where $\alpha\in E_T^k$. If $\alpha=((s_1,y_1),\cdots,(s_k,y_k))$ is a string of $k$ pairs in $E_T^k$, then $(\hat s,\hat y)$ refers to the pair with the largest first coordinate, i.e., $\hat s=s_1\vee\cdots\vee s_k$. By eliminating $(\hat s,\hat y)$ from $\alpha$ we obtain $\hat\alpha$. We refer the reader to the section~\ref{lemmas} for more details on the notations.
\end{thm}

We have a sequence of random functions $v_n$, defined through the Picard iteration~\eqref{picard}, which converge to $u$ in $\|\cdot\|_{\beta, p}$, for $\beta$ sufficiently large, and consequently in $L^p$, for all $(t,x)\in E_T$ and $p\geq 1$. We would like to use Theorem~\ref{bddD} to show that $u\in \mathbf D^{k,p}$ for all $p\geq 1$ and $k=1,2,\cdots$. Therefore the first step is to show that $v_n$'s are in $\mathbf D^{k,p}$ for all $k$ and $p$.  Then we need to show that $D^kv_n$'s are convergent weakly for all $k\geq1$; i.e; it is sufficient to have 
$$\sup_n\|v_n(t,x)\|_{k,p}<\infty,$$
where the $\|\cdot\|_{k,p}$ is the norm of $\mathbf D^{k,p}$ space. To this end we need a quantitive bound on the growth of Malliavin derivatives as well. We carry out this task by induction. The case $n=1$; i.e., $v_n(t,x)\in \mathbf D^{1,p}$ is the subject of subsection~\ref{first derivative}.  The second subsection is devoted to introducing a few notations and some technical lemmas that will allow us to go from the first derivative to the higher-order derivatives in the upcoming subsection.  The third subsection deals with the $k$'th derivatives of $v_n$'s and the short final subsection concludes the this section with proving main result of this section i.e., Theorem~\ref{4.1}. 

\subsection{The first derivative}\label{first derivative}
\begin{prp}\label{D1}
If the $v_n$'s are defined by~\eqref{picard}, then $v_n\in \mathbf D^{1,p}$ for all $n\geq0$, and: 
\begin{enumerate}
\item $Dv_0=0$ and 
\begin{align}\label{d1}
D_{s,y}v_{n+1}(t,x)=&q_{t-s}(x,y)\sigma(v_n(s,y))\\&+\intt\intT q_{t-r}(x,z)D_{s,y}\sigma(v_n(r,z))\-w(dr,dz).\notag
\end{align}
\item For all $n\geq0$, and $T\in[0,\infty),$
\begin{equation}\label{D1bound}
\|\G^1v_{n+1}\|^2_{\beta,p,T}\leq C_p\lip_{\sigma'}\Upsilon(2\beta/p)\left(1+\|\G^1v_n\|_{\beta,p,T}^2\right),
\end{equation}
where $\G^1_{t,x}f=\|Df(t,x)\|_H$, and 
$$\lip_{\sigma'}=2\max\left\{\sup_{x\in\T}|\sigma(x)|^2,\sup_{x\in\T}|\sigma'(x)|^2\right\}.$$
\end{enumerate}
\end{prp}
\begin{proof} 
We need to mention that $D_{s,y}\sigma(v_n(r,z))=0$ when $r<s$; i.e., the integral vanishes on the subinterval $[0,s]$.

We proceed by applying induction on $n$. We will find that the proofs of these three conclusions go hand in hand, i.e., we use~\eqref{D1bound} for $n=n_0$ to show that $v_{n_0+1}$ is in $\mathbf D^{1,2}$, with a derivative which satisfies~\eqref{d1}. Then we use~\eqref{d1}, to show~\eqref{D1bound} holds for $n=n_0+1$. 

For $n=0$ all these three conclusions hold vacuously. Next assume~\eqref{D1bound} holds for $v_n$, then $v_{n+1}\in \mathbf D^{1,2}$ and~\eqref{d1} holds by Proposition~\ref{derivative}. Here we will not go through the details to show this, as we will do this later for the derivatives of higher order [see subsection~\ref{lemmas} below]. Next, we show that~\eqref{D1bound} holds for $v_{n+1}$. Since  
\begin{align}
D_{s,y}v_{n+1}(t,x)&=q_{t-s}(x,y)\sigma(v_n(s,y))\\&+\intt\intT q_{t-r}(x,z)D_{s,y}\sigma(v_n(r,z))\-w(dr,dz),\notag
\end{align}
then by the triangular inequality for the $H$ norm
\begin{align*}
\G^1_{t,x}v_{n+1}&\leq
\sup_{x\in\T}|\sigma(x)|\left(\intt\|q_s\|_{L^2(\T)}^2\-ds\right)^{1/2}\\
&+\left(\intt\intT \left|\intt\intT q_{t-r}(x,z)D_{s,y}\sigma(v_n(r,z))\-w(dr,dz)\right|^2\-dy\-ds\right)^{1/2}.
\end{align*}

Take the $L^p(\Omega)$ norm from the both sides. Since, by the chain rule 
$$D_{s,y}\sigma(v_n(r,z))=\sigma'(v_n(r,z))D_{s,y}v_n(r,z),$$ and by the Burkholder's inequality for the Hilbert space-valued martingales,
\begin{align*}
&\E\left\|\intt\intT q_{t-r}(x,z)D_{s,y}\sigma(v_n(r,z))\-w(dr,dz)\right\|_H^p\\
&\lesssim_p\E\left(\intt\intT q^2_{t-r}(x,z)\|D_{s,y}\sigma(v_n(r,z))\|_H^2dz\-dr\right)^{p/2}\\
&\leq\sup_{x\in\T}|\sigma'(x)|^p\E\left(\intt\intT q^2_{t-r}(x,z)\|D_{s,y}v_n(r,z)\|_H^2dz\-dr\right)^{p/2},
\end{align*}
we obtain
\begin{eqnarray*}
\|\G^1_{t,x}v_{n+1}\|_{L^p(\Omega)}\lesssim_p
\sup_{x\in\T}|\sigma(x)|\left(\intt\|q_s\|_{L^2(\T)}^2\-ds\right)^{1/2}\\
+\sup_{x\in\T}|\sigma'(x)|\left\{\E\left(\intt\intT q^2_{t-r}(x,z)|\G^1_{r,z}v_n|^2dz\-dr\right)^{p/2}\right\}^{1/p}.
\end{eqnarray*}
Now, apply Minkowski's inequality $[d\P\times\-dz\-dr]$ to the last term to switch the expectation and the double integral, as follows
\begin{align*}
\|\G^1_{t,x}&v_{n+1}\|^2_{L^p(\Omega)}\\
&\lesssim_{p,\sigma'}
\left(\intt\|q_s\|_{L^2(\T)}^2\-ds
+\intt\intT q^2_{t-r}(x,z)\|\G^1_{r,z}v_n\|_{L^p(\Omega)}^2dz\-dr\right)\\
&\lesssim_{p,\sigma'}\left(\intt\|q_s\|_{L^2(\T)}^2\-ds
+\|\G^1v_n\|_{\beta,p, T}^2\intt\intT q^2_{t-r}(x,z) e^{2\beta r/p}dz\-dr\right),
\end{align*}
 where the last inequality follows from the trivial inequality 
 \begin{equation}\label{trivial}
\|\G^1_{r,z}v_n\|_{L^p(\Omega)}^2\leq e^{2\beta r/p}\|\G^1v_n\|_{\beta,p, T}^2\hspace{.2 in}(r,z)\in[0,T]\times\T.
\end{equation}
Therefore if we change the variable $t-r\to r$ in the second integral, and multiply the inequality by $e^{-2\beta t/p}$ we arrive at 
\begin{eqnarray*}
e^{-2\beta t/p}\|\G^1_{t,x}v_{n+1}\|^2_{L^p(\Omega)}\hspace{3 in}\\
\lesssim_p
 \lip_{\sigma'}\left(e^{-2\beta t/p}\intt\|q_s\|_{L^2(\T)}^2\-ds
+\|\G^1v_n\|_{\beta,p}^2\intt\|q_r\|_{L^2(\T)}^2 e^{-2\beta r/p}\-dr\right)\\
\lesssim_p\lip_{\sigma'}\Upsilon(2\beta/p)\left(1+\|\G^1v_n\|_{\beta,p,T}^2\right),
\end{eqnarray*}
where the last inequality follows from Lemma~\ref{supineq}.  By optimizing this expression over all $t\in[0,T]$ and all $x\in\T$ we arrive at
\begin{equation}\label{G^1bd}
\|\G^1v_{n+1}\|_{\beta,p,T}^2\leq C_p\lip_{\sigma'}\Upsilon(2\beta/p)\left(1+\|\G^1v_n\|_{\beta,p,T}^2\right).
\end{equation}
This proves that $\|Dv_{n+1}\|_H\in L^p(\Omega)$, and so $v_{n+1}(t,x)\in \mathbf D^{1,p}$, for all $(t,x)\in E_T$ by Proposition~\ref{22p}. 
\end{proof}
\begin{prp}\label{D1pu}
If $u=u(t,x)$ be the solution to the Eq.~\eqref{intequ}, then $u(t,x)\in \mathbf D^{1,p}$ and $Du$ satisfies~\eqref{Du}; i.e.,
\begin{equation}\label{D1u}
D_{s,y} u(t,x)=q_{t-s}(x,y)u(s,y)+\intt\intT q_{t-r}(x,z)D_{s,y}\sigma(u(r,z))\-w(dr,dz).
\end{equation}
\end{prp}
\begin{proof}
The fact that $u\in \mathbf D^{1,2}$ follows easily from the bound~\eqref{G^1bd}, because by iterating this bound we get
 \begin{eqnarray*}
 \|\G^1v_n\|^2_{\beta, p}\leq\alpha+\alpha^2+\cdots+\alpha^n,
  \end{eqnarray*}
where $\alpha:=C_p\lip_{\sigma'}\Upsilon(2\beta/p).$ Since $\lim_{\beta\to\infty}\Upsilon(\beta)=0$, we can choose $\beta>0$ sufficiently large so that $\alpha<1$ and consequently we get
 \begin{eqnarray}\label{Tfinite}
\sup_n \|\G^1v_n\|^2_{\beta, p,T}\leq\frac{\alpha}{1-\alpha}<\infty.
  \end{eqnarray}
Then $u(t,x)\in \mathbf D^{1,2}$, by Proposition ~\ref{bddD}. Then Proposition~\ref{22p} proves that $u(t,x)\in \mathbf D^{1,p}$ for all $p\geq 1$.  Since the right-hand-side of ~\eqref{Tfinite} is independent of $T$, it holds for $T=\infty$. Because of this, in the rest of the paper, we only work with the norm $\|\cdot\|_{\beta, p}$.

In order to derive ~\eqref{D1u}, it suffices to show that
\begin{equation}
\|\G^1(v_n-u)\|_{\beta,p}\to 0,\hspace{.2 in}\text{as}\;n\to\infty,
\end{equation}
where 
\begin{equation*}
\G_{t,x}^1u:=\|Du(t,x)\|_{H}.
\end{equation*}
By the triangular inequality, applied first to the $H$ norm and then to the $L^p(\Omega)$ norm, and squaring both sides, we have 
\begin{align*}
&\frac{1}{2}\|\G^1(v_{n+1}-u)\|^2_{L^p(\Omega)}
\leq\left\{\E\|q_{t-\bullet}(x,\ast)[\sigma(v_n(\bullet,\ast))-\sigma(u(\bullet,\ast))]\|^p_H\right\}^{2/p}\hspace{.3 in}\\
&+\left\{\E\left\|\intt\intT q_{t-r}(x,z)D[\sigma(v_n(r,z))-\sigma(u(r,z))]w(dr, dz)\right\|^p_H\right\}^{2/p}.
\end{align*}
Since $\sigma$ is Lipschitz, by applying Minkowski's inequality to the first term and Burkholder's inequality followed by Minkowski's inequality to the second term, we get 
\begin{eqnarray*}
C\|\G^1_{t,x}(v_{n+1}-u)\|^2_{L^p(\Omega)}\leq\intt\intT q^2_{t-s}(x,y)\left\|v_n(s,y)-u(s,y)\right\|^2_{L^p(\Omega)}\-dy\-ds
\hspace{.3 in}\\
+\intt\intT q^2_{t-r}(x,z)\left\|\G^1_{r,z}(\sigma(v_n)-\sigma(u))\right\|_{L^p(\Omega)}^2 dz\-dr.
\end{eqnarray*}
Then, by ~\eqref{trivial} and Lemma~\ref{supineq} we obtain
\begin{eqnarray*}
C\|\G^1_{t,x}(v_{n+1}-u)\|^2_{L^p(\Omega)}\leq\left\|v_n-u\right\|^2_{\beta,p}e^{2\beta t/p}\Upsilon(2\beta/p)
\hspace{.3 in}\\
+\left\|\G^1(\sigma(v_n)-\sigma(u))\right\|_{\beta,p}^2e^{2\beta t/p}\Upsilon(2\beta/p),
\end{eqnarray*}
where $C$ depends on $\lip_{\sigma'}$, and $C_p$, the constant in Burkholder's inequality. After we optimize on $t\in[0,T]$ and $x\in[0,2\pi]$, we have 
\begin{equation}\label{rineq}
\|\G^1(v_{n+1}-u)\|_{\beta,p}^2\leq C_{\beta,p,\sigma'}\left(\|v_n-u\|_{\beta,p}^2+\left\|\G^1(\sigma(v_n)-\sigma(u))\right\|_{\beta,p}^2\right),
\end{equation}
where $C_{\beta,p,\sigma'}\to0$ as $\beta\to\infty$. Since
\begin{align*}
D[\sigma(v_n(r,z))-\sigma(u(r,z))]=\sigma'(v_n(r,z))[Dv_n(r,z)-Du(r,z)]\\+[\sigma'(v_n(r,z)-\sigma'(u&(r,z))]Du(r,z),
\end{align*}
then by applying the triangular inequality, and considering the boundedness and the Lipschitz property of $\sigma'$ we obtain the following:
\begin{align*}
C\G^1_{r,z}(\sigma(v_n)-\sigma(u))\leq\G^1_{r,z}(v_n-u)+|v_n(r,z)-u(r,z)|\G^1_{r,z}u.
\end{align*}
Therefore,
\begin{align}\label{D1ua}
C\|\G^1_{r,z}(\sigma(v_n)-\sigma(u))\|_{\beta,p}\leq\|\G^1(v_n-u)\|_{\beta,p}+\|(v_n-u)\G^1u\|_{\beta,p}.
\end{align}
By optimizing over $(t,x)\in E_T$ and substituting in~\eqref{rineq} we obtain
\begin{align}\label{D13}
\|\G^1(v_{n+1}-u)\|^2_{\beta,p}\leq C_{\beta,p,\sigma'}(\|v_n-u\|_{\beta,p}^2&+\|(v_n-u)\G^1u\|^2_{\beta,p}\\&+\|\G^1(v_n-u)\|^2_{\beta,p}).\notag
\end{align}
Consider the first two terms in the parenthesis in~\eqref{D13}. From Thorem~\ref{existence} we know that $\|v_n-u\|_{\beta,p}^2\to0$ as $n\to\infty$, while the second term in~\eqref{D13} vanishes as $n\to\infty$, for example, by the Cauchy-Schwarz inequality. Therefore~\eqref{D13} can be written as 
\begin{equation}\label{iterate2}
\|\G^1(v_{n+1}-u)\|^2_{\beta,p}\leq C_{\beta,p,\sigma'}\left(\lambda_n+\|\G^1(v_n-u)\|^2_{\beta,p}\right),
\end{equation}
where $\lambda_n\to0$ as $n\to\infty$. Choose $\beta$ sufficiently large such that $C_{\beta,p,\sigma'}<1$. Then ~\eqref{iterate2} implies that
$$\lim_{n\to\infty}\|\G^1(v_{n+1}-u)\|^2_{\beta,p}=0.$$
This finishes the proof.
\end{proof}
To state and prove the result for derivatives of higher order, we need to introduce some notations and prove some preparatory lemmas that will be stated next.

\subsection{Preliminaries}\label{lemmas}
We know that  the $m$th derivative of $v_n(t,x)$, if exists, belongs to $L^{2}(E_T^{m+1}\times\Omega)$. Recall that $\L^{1,2}=\mathbf D^{1,2}(E_T^{m+1}\times\Omega)$. 
Let $\alpha=\alpha_m$ denote an element in $E_T^m$. We can write 
$$\alpha=((s_1,y_1),\cdots,(s_m,y_m)).$$

Let $\ss:=\max\{s_1,\cdots,s_m\}$. If $i\in\{1,\cdots,m\}$ is so that $s_i=\ss$, then we let $\yy$ denote $y_i$ and $\al_m:=((s_1,y_1),\cdots,(s_{i-1},y_{i-1}),(s_{i+1},y_{i+1}),\cdots,(s_m,y_m))$. Note that $\al_m\in E_T^{m-1}$. 

If we think of $\alpha=\alpha_m$ as a set of $m$ pairs $\alpha=\{(s_1,y_1),\cdots,(s_m,y_m)\}$, instead of an ordered $m$-tuple, then the partitions of $\alpha$ are defined. Let $\mathcal P^m:=\text{the set of all partitions of}\;\alpha\ E_T^m$.

If $\d=\{\d_1,\cdots, \d_l\}\in\mathcal{P}^m$, then let $|\d_j|$ denote the cardinality of $\d_j$, where $j=1,\cdots,l$. Clearly $|\d_1|+\cdots+|\d_l|=m$. If  $\d=\{\d_1,\cdots, \d_l\}\in\mathcal{P}^m$, then $D^{|\d_j|}_{\d_j}F$ makes sense for $j=1,\cdots,l$. For example if $\d_1=\{(s_1,y_1),(s_3,y_3)\}$, then $D^{|\d_1|}_{\d_1}F=D^2_{(s_1,y_1)(s_3,y_3)}F$. Furthermore, if $\d=\{\d_1,\cdots,\d_l\}\in\mathcal{P}^m$, then we introduce the new notation $D^\d F$, and define it by
\begin{equation}\label{rhonotation}
D_\alpha^\d F:=D^{|\d_1|}_{\d_1}F\times\cdots\times D^{|\d_l|}_{\d_l}F.
\end{equation}
Notice that 
\begin{equation}\label{Htensornorm}
\|D^\rho F\|_{H^{\otimes m}}=\Pi_{i=1}^l \|D^{|\d_i|}F\|_{H^{\otimes|\d_i|}}.
\end{equation}
Fix $l\leq m$ and let $\mathcal{P}_l^m$ denote the set of all $\d\in\mathcal{P}^m$ such that $|\d|=l$. Trivially 
$$\mathcal{P}^m=\cup_{l=1}^m\mathcal{P}^m_l.$$
We let $\G^{\d}_{x,y}v$ to denote the $H^m$
norm of $D^\rho v(t,x)$ i.e
\begin{equation}\label{Gamtx}
\G^\d_{t,x}v=\|D^{\d}v(t,x)\|_{H^{\otimes m}}.
\end{equation}

If  $\d$ denotes the only member of $\mathcal{P}^m_1$, i.e., $\d=\{\{(s_1,y_1),\cdots,(s_k,y_k)\}\}$, then write $\G^m_{t,x}v$ instead of $\G^{\d}_{t,x}v$.

The following lemma allows us to approximate $\|\sigma(v_n(t,x)\|_{k,p}$, where $\|\cdot\|_{k,p}$ denotes the norm on $\mathbf D^{k,p}$. 
\begin{lma}
Assume $\sigma$ is smooth and bounded together with all its derivatives. If $F\in \cap_{p\in[1,\infty)}\mathbf D^{m,p}$, then so is $\sigma(F)$. Furthermore, for $\alpha\in E_T^m$,
\begin{equation}\label{mthD}
D^m_\alpha\sigma(F)=\sum_{j=1}^m\sigma^{(j)}(F)\sum_{\d\in\mathcal{P}_j^m} D_\alpha^\d F,
\end{equation}
where $\mathcal{P}_j^m$ is the set of all partitions of $\alpha$, comprised of $j$ components $\d_1,\cdots,\d_j$, and $\sigma^{(j)}$ denotes the $j$th derivative of $\sigma$.
\end{lma}

\begin{proof}We can easily prove this for the smooth functionals by induction, and then extend the result to $F\in\cap_{p\in[1,\infty)}\mathbf D^{m,p}$ by approximation by the smooth functionals.
\end{proof}
\begin{lma}\label{D^kbp}
Let $\alpha\in E_T^m$, and $\d=\{\d_1,\cdots,\d_l\}\in\mathcal{P}^m$ denote a partition of $\alpha$. Let $U(t,z)$ and $V(t,z)$ belong to $\cap_{p>1}\mathbf D^{m,p}$ for almost all $r,z$ and $\|\G^{|\d_j|}V\|_{\beta,p}<\infty$ for all $p\geq1$. If $\sigma$ is bounded and  smooth with bounded derivatives of all orders, then 
 \begin{align}\label{Gsigma}
 \|\G^m\sigma(V)\|_{\beta,p}\lesssim
\|\G^mV\|_{\beta,p}+\sum_{j=2}^m\sum_{\d_1\cdots\d_j\in\mathcal{P}^m}\Pi_{i=1}^j\|\G^{|\d_j|}V\|_{\beta,jp}.
\end{align}
Furthermore, if $\|\G^{|\d_j|}U\|_{\beta,p}<\infty$ for all $p\geq1$, then we have 
\begin{equation}\label{Gssigma}
\|\G^m(\sigma(V)-\sigma(U))\|_{\beta,p} \lesssim
\|\G^m(V-U)\|_{\beta,p}+\sum_{j=2}^m\sum_{\d_1\cdots\d_j\in\mathcal{P}^m}\Pi_{i=1}^j\|\G^{|\d_j|}(V-U)\|_{\beta,jp}.
\end{equation}
\end{lma}
\begin{proof} 
According to ~\eqref{mthD} we have
\begin{align*}
\G^m_{r,z}\sigma(V)\leq C\sum_{j=1}^m\sum_{\d_1\cdots\d_j\in\mathcal{P}^m}\Pi_{i=1}^j\|D^{|\d_i|}V(r,z)\|_{H^{\otimes|\d_i|}},
\end{align*}
where $C=\sup_x\{\sigma(x),\sigma'(x),\cdots,\sigma^{m}(x)\}$. Then
\begin{align*}
\|\G^m_{r,z}\sigma(V)\|_{L^p(\Omega)}\lesssim \|\G^m_{r,z}V\|_{L^p(\Omega)}+\sum_{j=2}^m\sum_{\d_1\cdots\d_j\in\mathcal{P}^m}\|\Pi_{i=1}^j\G^{|\d_i|}_{r,z}V\|_{L^p(\Omega)}.
\end{align*}
Therefore, by the generalized H\"older inequality, 
\begin{align*}
\|\G^m_{r,z}\sigma(V)\|_{L^p(\Omega)}\lesssim \|\G^m_{r,z}V\|_{L^p(\Omega)}+\sum_{j=2}^m\sum_{\d_1\cdots\d_j\in\mathcal{P}^m}\Pi_{i=1}^j\|\G^{|\d_i|}_{r,z}V\|_{L^{jp}(\Omega)}.
\end{align*}
Multiplying both sides by $e^{\beta r/p}$ we get 
\begin{align*}
e^{\beta r/p}\|\G^m_{r,z}&\sigma(V)\|_{L^p(\Omega)}\\
&\lesssim
e^{\beta r/p}\|\G^m_{r,z}V\|_{L^p(\Omega)}+
\sum_{j=2}^m\sum_{\d_1\cdots\d_j\in\mathcal{P}^m}\Pi_{i=1}^je^{\beta r/jp}\|\G^{|\d_i|}_{r,z}V\|_{L^{jp}(\Omega)}\\
&\lesssim\|\G^mV\|_{\beta,p}+\sum_{j=2}^m\sum_{\d_1\cdots\d_j\in\mathcal{P}^m}\Pi_{i=1}^j\|\G^{|\d_j|}V\|_{\beta,jp}.
\end{align*}
Therefore,
\begin{align}\label{previous}
\|\G^m\sigma(V)\|_{\beta,p}\leq C\left(
\|\G^mV\|_{\beta,p}+\sum_{j=2}^m\sum_{\d_1\cdots\d_j\in\mathcal{P}^m}\Pi_{i=1}^j\|\G^{|\d_j|}V\|_{\beta,jp}\right).
\end{align}
The proof of the second statement is similar, if we observe that,
\begin{eqnarray*}
D(\sigma(V)-\sigma(U))=\sum_{j=1}^m\sum_{\d\in\mathcal{P}^m_j}
\sigma^{j}(V)(D^{\d_j}(V-U))+[\sigma^{j}(V)-\sigma^{j}(U)]D^{\d_j}U.
\end{eqnarray*}
In this case, the constant $C$ in~\eqref{previous}, is replaced by $C'=C\vee\lip_{\sigma'}\vee\cdots\vee\lip_{\sigma^{(m)}}$. 
\end{proof}
The following lemma explains the method that we use to prove a random variable is in $\mathbf D^{k+1,p}$, when we know it is in $\mathbf D^{k,p}$.
\begin{lma}
Let $F\in \mathbf D^{k,p}$ satisfy  $D^k_\alpha F\in \mathbf D^{1,p}$ for almost all $\alpha\in E_T$. If $\E\|DD^kF\|_{H^{\otimes k+1}}<\infty$, then $F\in D^{k+1,p}$.
\end{lma}
\begin{proof}
To make the notation simpler, we prove the lemma only for $k=$ and $p=2$. In this case we have $F\in \mathbf D^{1,2}$ and $DF\in \mathbf D^{1,2}(L^2(T))$. By Proposition~\ref{iff}
we need only to show that 
\begin{equation}\label{n(n-1)}
\sum_{n=1}^\infty n(n-1)\|J_nF\|_{L^2(\Omega)}^2<\infty.
\end{equation}
Since $DF\in \mathbf D^{1,2}(L^2(T))$, then 
\begin{eqnarray*}
\sum_{n=1}^\infty n\|J_nDF\|_{L^2(\Omega\times T)}^2<\infty.
\end{eqnarray*}
Since $F\in \mathbf D^{1,2}$, then $\langle DJ_nF,h\rangle_{L^2(T)}=J_{n-1}(\langle DF,h\rangle_{L^2(T)})$, then
\begin{equation*}
\|J_nDF\|_{L^2(\Omega\times T)}^2=\E\|J_nDF\|_{L^2(T)}^2=\E\|DJ_{n+1}F\|_{L^2(T)}^2=(n+1)\|J_nF\|_2^2.
\end{equation*}
For the proof of the last equality we refer the reader to \cite{cbms}, Proposition 1.12.
\end{proof}
\begin{lma}\label{lemf}
Let $V(t,z)\in\cap_{p>1}D^{m+1,p}$ for almost all $r,z$ and let $\|\G^{|\d_j|}V\|_{\beta,p}<\infty$ for all $p>1$, where $\d=\{\d_1,\cdots,\d_l\}\in\mathcal{P}^{m+1}$. For $\alpha\in\E_T^m$ let 
\begin{equation}
f_\alpha(r,z)=q_{t-r}(x,z)D^m_\alpha\sigma(V(r,z)).
\end{equation}
Then $f_\alpha\in\L^{1,2}$.
\end{lma}
\begin{proof}
We need to verify that the three conditions mentioned in Remark~\ref{L^1,2} hold for $f_\alpha$.\begin{enumerate}
\item By Lemma~\ref{D^kbp}, $\|\G^m\sigma(V)\|_{\beta,2}<\infty$. Then
\begin{align*}
\int_{E_T^m}\E\|f_\alpha\|^2_{L^2(E_T\times\Omega)}d\alpha
&=\E\int_0^t\intT q^2_{t-r}(x,z)\int_{E_T^m}|D^m_\alpha\sigma(V(r,z))|^2d\alpha\-dz\-dr\\
&=\int_0^t\intT q^2_{t-r}(x,z)\E[\G^m_{r,z}\sigma(V)]^2dz\-dr\\
&\leq\int_0^t\intT q^2_{t-r}(x,z)e^{\beta r}\|\G^m\sigma(V)\|_{\beta,2}^2dz\-dr<\infty.
\end{align*}
This also means that $\|f_\alpha\|^2_{L^2(E_T\times\Omega)}<\infty$ for almost all $\alpha\in E_T$.
\item $f_\alpha(r,z)\in \mathbf D^{1,2}$ because $V(t,x)\in D^{m+1,2}$
\item Since $\|\G^{m+1}\sigma(V)\|_{\beta,2}<\infty$, then 
\begin{align}\label{Dfalpha}
&\int_{E_T^m}\E\|Df_\alpha\|^2_{L^2(E_T^2\times\Omega)}d\alpha\\
&=\E\int_0^t\intT q^2_{t-r}(x,z)\int_{E_T^{m+1}}|D_{s,y}D^m_\alpha\sigma(V(r,z))|^2d\lambda\-dz\-dr\notag\\
&=\int_0^t\intT q^2_{t-r}(x,z)\E[\G^{m+1}_{r,z}\sigma(V)]^2dz\-dr\notag\\
&\leq\int_0^t\intT q^2_{t-r}(x,z)e^{\beta r}\|\G^{m+1}\sigma(V)\|_{\beta,2}^2dz\-dr<\infty,\notag
\end{align}
where $d\lambda=d\alpha\-dy\-ds$. The last result also shows that $Df_\alpha\in L^2(E_T^2\times\Omega)$ for almost all $\alpha\in E_T$.
\end{enumerate}
Therefore $f_\alpha\in\L^{1,2}$ for almost all $\alpha$.
\end{proof}
\begin{lma}\label{domDf}
If $V$ and $f_\alpha$ are as defined in Lemma~\ref{lemf}, and satisfy the same conditions, then $D_{s,y}f_\alpha\in\Dom\-\delta$.
\end{lma}
\begin{proof}
Applying Fubini's theorem to ~\eqref{Dfalpha} yields, 
\begin{equation*}
\|D_{s,y}f_\alpha\|_{L^2(E_T\times\Omega)}<\infty\hspace{.2 in}\text{for almost all}\;((s,y),\alpha)\in E_T^{m+1}.
\end{equation*}
Since $D_{s,y} f_\alpha$ is adapted and belongs to $L^2(E_T\times\Omega)$ for almost all $(s,y)$ and $\alpha$, then the It\^o integral of $D_{s,y}f_\alpha$ is defined and coincides with $\delta(f_\alpha)$. 
\end{proof}
\begin{lma}\label{intdom}
If $V$ and $f_\alpha$ are as defined in Lemma~\ref{lemf}, and satisfy the same conditions, then for each $\alpha\in E_T^m$
\begin{equation}
\int_0^t\intT Df_\alpha(r,z)\-w(dr,dz)\in L^2(E_T\times\Omega).
\end{equation}
\end{lma}
\begin{proof}
This follows from Burkholder's inequality:
\begin{align*}
&\E\int_0^t\intT\left|\int_0^t\intT Df_\alpha(r,z)\-w(dr,dz)\right|^2dy\-ds\\
&\leq\E\int_0^t\intT q^2_{t-r}(x,z)\|DD^m_\alpha \sigma(v_N(r,z))\|_H^2\-dz\-dr.
\end{align*}
To show that the last expectation is finite for almost all $\alpha\in E_T^m$, we take integral with respect $\alpha$, and then Fubini's theorem implies that
\begin{align*}
&\int_{E_T^m}\E\int_0^t\intT q^2_{t-r}(x,z)\|DD^m_\alpha\sigma(v_N(r,z))\|_H^2\-dz\-dr\-d\alpha\\
&\leq\E\int_0^t\intT q^2_{t-r}(x,z)\|D^{m+1} \sigma(v_N(r,z))\|_{H^{\otimes m+1}}^2\-dz\-dr\\
&\leq\int_0^t\intT q^2_{t-r}(x,z)e^{\beta r}\|\G^{m+1}\sigma(v_N)\|^2_{\beta,2}\-dz\-dr<\infty.
\end{align*}
Among other things, this proves that the integrand is finite for almost all $\alpha$.
\end{proof}
\begin{lma}\label{lastlem}
 Let $V$ and $f_\alpha$ be as defined in Lemma~\ref{lemf}, and satisfy the same conditions.
 Define 
 \begin{equation}\label{F_1}
 F_1(\alpha)=\int_0^t\intT f_\alpha(r,z)\-w(dr,dz).
 \end{equation}
Then:
 \begin{enumerate}
\item$F_1(\alpha)\in \mathbf D^{1,2}$ for almost all $\alpha\in E_T^m$;
\item $DF_1$ is given by
\begin{equation}\label{D^kF}
D_{s,y}F_1(\alpha)=f_\alpha(s,y)+\int_0^t\intT D_{s,y}f_\alpha(r,z)\-w(dr,dz);
\end{equation}
\item We have  
\begin{equation}\label{E|DF|p}
\E\left(\|D F_1\|_{H^{\otimes m+1}}^p\right)<\infty.\end{equation}
\end{enumerate}
\end{lma}
\begin{proof}
After proving Lemma~\ref{lemf},~\ref{domDf} and~\ref{intdom}, we know that $F_1(\alpha)$ satisfies the assumptions of Proposition~\ref{derivative}. Therefore, it is an immediate consequence of Proposition~\ref{derivative} that $F_1(\alpha)\in\mathbf D^{1,2}$ and~\eqref{D^kF} holds. We finally prove~\eqref{E|DF|p} as follows. By the Burkholder's inequality,
\begin{align*}
\E\left(\|F_1\|_{H^{\otimes m+1}}^p\right)&=\E\left\|\int_0^t\intT q_{t-r}(x,z)D^{m+1}\sigma(V(r,z))\-w(dr,dz)\right\|_{H^{\otimes m+1}}^p\\
&\leq C_p\E\left(\int_0^t\intT q^2_{t-r}(x,z)\|D^{m+1}\sigma(V(r,z))\|^2_{H^{\otimes m+1}}\-dz\-dr\right)^{p/2}.
\end{align*}
Therefore, by Minkowski's inequality, 
\begin{align*}
\left\{\E\left(\|F_1\|_{H^{\otimes m+1}}^p\right)\right\}^{2/p}\leq C_p^{2/p}\int_0^t\intT q^2_{t-r}(x,z)\{\E|\G_{r,z}^{m+1}\sigma(V)|^p\}^{2/p}\-dz\-dr.
\end{align*}
Therefore,
\begin{align}\label{F_1p}
\left\{\E\left(\|F_1\|_{H^{\otimes m+1}}^p\right)\right\}^{2/p}\leq C_p^{2/p}e^{2\beta t/p}\|\G^{m+1}\sigma(V)\|_{\beta,p}^2\int_0^t\|q_\tau\|_{L^2(T)}^2e^\frac{-2\beta\tau}{p}\-dr.
\end{align}
After rearranging and choosing a new constant, we arrive at
\begin{align}\label{F_1betap}
\left\{e^{-\beta t}E\|F_1\|_{H^{\otimes m+1}}^p\right\}^{1/p}\leq C_p
\|\G^{m+1}\sigma(V)\|_{\beta,p}\sqrt{\Upsilon(2\beta/p)}.
\end{align}
This ends the proof.
\end{proof}
\begin{rmk}\label{lastrmk}
If we define a random variable $\tilde F_1:=\|F_1\|_{H^{\otimes m+1}}$, then the ~\eqref{F_1betap} can be written as 
\begin{equation}
\|\tilde F_1\|_{\beta,p}\leq C_p\|\G^{m+1}\sigma(V)\|_{\beta,p}\sqrt{\Upsilon(2\beta/p)}.
\end{equation}
\end{rmk}

\subsection{The $k$th derivatives of the $v_n$'s and the smoothness of the solution}
\begin{prp}\label{Gv_n}
Let $v_n$ be defined by~\eqref{picard} for $n=0,1,\cdots$, where $\sigma\in C^\infty_b(\R)$ and $q_t(x)$ satisfies hypothesis {\bf{H\-\ref{condition0}}}. Then $v_n\in \mathbf D^{k,p}$ for $k=1,2,\cdots$ and $p\geq2$. Furthermore if $\alpha\in E_T^k$, then:
\begin{enumerate}
\item $D^kv_0=0$ and 
\begin{align}\label{D^kv.}
D^k_\alpha v_{n+1}(t,x)&=q_{t-\ss}(\yy,x)D^{k-1}_{\al}\sigma(v_n(\ss,\yy))\\&+\int_0^t\intT q_{t-r}(x,z)D^k_{\alpha}\sigma(v_n(r,z))\-w(dr,dz);\notag
\end{align}
\item For some $C>0$ which only depends on $p$: 
\begin{equation}\label{recursive}
\|\G^{m+1}v_{n+1}\|^2_{\beta,p}\leq C\Upsilon(2\beta/p)(1+\|\G^{m+1}v_n\|^2_{\beta,p});
\end{equation}
\item If $\d=\{\d_1,\cdots,\d_l\}$, then 
\begin{align}\label{d}
\|\G^{\d} v_{n+1}\|_{\beta,p}\leq\Pi_{j=1}^l\|\G^{|\d_j|}v_{n+1}\|_{l\beta,lp}<\infty.
\end{align}
\end{enumerate}
\end{prp}

\begin{proof}
We proceed by applying induction on $n$ and $k$. When $k=1$ and  $\alpha=(s,y)=\al$, we have shown in Proposition~\ref{D1} that all above claims hold.   Next, by assuming that the claims hold for all $n\geq0$ and $k=1,\cdots,m$, we will prove that they also hold for $m+1$ and all $n\geq0$. Since $D^{m+1}v_0=0$, ~\eqref{recursive} and~\eqref{d} hold for $n=0$. Suppose the claims hold for $n=0,\cdots,N$. To prove the claims for $N+1$, notice that by Lemma~\ref{D^kbp},
$$\|\G^{m+1}\sigma(v_N)\|_{\beta,p}<\infty.$$
Then by Lemma~\ref{lastlem}, 
$\intt\intT q_{t-r}(x,z)D^k_\alpha\sigma(v_n(r,z))\-w(dr,dz)$ belongs to $\mathbf D^{1,2}$.  If we let $\gamma=((s,y),\alpha)\in E_T^{m+1}$, then after relabeling $\gamma$, we have  
$$\gamma=((s_1,y_1),\cdots,(s_{m+1},y_{m+1})).$$ We let $\ss=\max\{s_1,\cdots,s_{m+1}\}$, and define $\hat\gamma$ and $\yy$ accordingly. Then by~\eqref{D^kF},
\begin{align}
D_\gamma^{m+1}v_{N+1}(t,x)=&\;q_{t-\ss}(x,\yy)
D^m_{\hat\gamma} v_N(\ss,\yy)\\
&+\intt\intT q_{t-r}(x,z)D^{m+1}_\gamma\sigma(v_N(r,z))\-w(dr,dz),\notag
\end{align}
where we applied the fact that $D_{s,y}f_\alpha(r,z)=0$ if $s>r$. By the triangular inequality for the $H^{\otimes m+1}$ norm, 
\begin{align*}
&\G^{m+1}_{t,x}v_{N+1}\leq\\
&\left(\sum_{j=1}^{m+1}\underbrace{\intt\intT\cdots\intt\intT}_{m+1\;\text{times}}q^2_{t-s_j}(x,y_j)\left|D^m_{\gamma_j}\sigma(v_N(s_j,y_j))\right|^2\1_{\ss=s_j}(\gamma)\-d\alpha_j\-dy_j\-ds_j\right)^{1/2}\\
&+\left\|\intt\intT q^2_{t-r}(x,z)D^{m+1}\sigma(v_N(r,z))\-w(dr,dz)\right\|_{H^{\otimes m+1}},
\end{align*}
where $\gamma_j=((s_1,y_1),\cdots,(s_{j-1},y_{j-1}),(s_{j+1},y_{j+1}),\cdots,(s_{m+1},y_{m+1}))$. Notice that $\hat\gamma=\gamma_j$ when $\ss=s_j$.
All the integrals inside the sum are equal, and by omitting the indicator function $\1_{\ss=s_j}(\alpha)$ we arrive at
\begin{eqnarray*}
\G^{m+1}_{t,x}v_{N+1}\leq
\left((m+1)\intt\intT q^2_{t-s_1}(x,y_1)\left\|D^m\sigma(v_N(s_1,y_1))\right\|_{H^{\otimes m}}^2\-dy_1\-ds_1\right)^{1/2}\\
+\left\|\intt\intT q^2_{t-r}(x,z)D^{m+1}\sigma(v_N(r,z)\-w(dr,dz)\right\|_{H^{\otimes m+1}}.
\end{eqnarray*}
Then, by the triangular inequality for the $L^p(\Omega)$ norm, followed by Burkholder's inequality applied to the second integral on the right-hand-side,
\begin{align*}
\|\G^{m+1}_{t,x}&v_{N+1}\|_{L^p(\Omega)}\\
&\leq
(m+1)^{1/2}\left\{\E\left(\intt\intT q^2_{t-s_1}(x,y_1)\left|\G_{s_1,y_1}^m\sigma(v_N)\right|^2\-dy_1\-ds_1\right)^{p/2}\right\}^{1/p}\\
&+C_p\left\{\E\left(\intt\intT q^2_{t-r}(x,z)\left|\G_{r,z}^{m+1}\sigma(v_N)\right|^2\-dz\-dr\right)^{p/2}\right\}^{1/p}.
\end{align*}
If we square both sides of the last inequality and then apply Minkowski's inequality to the both integrals on the right-hand-side, then we obtain
\begin{eqnarray*}
A_p\left\|\G^{m+1}_{t,x}v_{N+1}\right\|^2_{L^p(\Omega)}\leq
\intt\intT q^2_{t-s_1}(x,y_1)\left\|\G_{s_1,y_1}^m\sigma(v_N)\right\|^2_{L^p(\Omega)}\-dy_1\-ds_1\\
+\intt\intT q^2_{t-r}(x,z)\left\|\G_{r,z}^{m+1}\sigma(v_N)\right\|^2_{L^p(\Omega)}\-dz\-dr,
\end{eqnarray*}
where $A_p=\frac{1}{2(m+1)\vee2C_p^2}$. By~\eqref{trivial} we have
\begin{align*}
A_p\left\|\G^{m+1}_{t,x}v_{N+1}\right\|^2_{L^p(\Omega)}\leq
&\left(\left\|\G^m\sigma(v_N)\right\|^2_{\beta,p}
+\left\|\G^{m+1}\sigma(v_N)\right\|^2_{\beta,p}\right)\\
&\times\intt\intT q^2_{t-r}(x,z)e^{2\beta r/p}\-dz\-dr.
\end{align*}
By optimizing on all $t>0$, for some constant $B_p>0$ which only depends on $p$, we have 
\begin{eqnarray*} 
\left\|\G^{m+1}v_{N+1}\right\|^2_{\beta,p}\leq
B_p\left(\left\|\G^m\sigma(v_N)\right\|^2_{\beta,p}
+\left\|\G^{m+1}\sigma(v_N)\right\|^2_{\beta,p}\right)\Upsilon(2\beta/p).
\end{eqnarray*}
Therefore, ~\eqref{Gsigma}, and the induction hypothesis~\eqref{d} for $n=N$, imply that $\left\|\G^m\sigma(v_N)\right\|^2_{\beta,p}<\infty$. Therefore, by choosing a constant $C>0$ sufficiently large, we obtain 
\begin{eqnarray*}
\left\|\G^{m+1}v_{N+1}\right\|^2_{\beta,p}\leq
C\left(1+\left\|\G^{m+1}v_N\right\|^2_{\beta,p}\right)\Upsilon(2\beta/p).
\end{eqnarray*}

This proves~\eqref{recursive} for $k=m+1$ and all $n\geq 0$, in the sense that $v_{N+1}(t,x)\in\mathbf D^{m+1,p}$, for all $p\geq 2$.

The proof is not complete yet, as we need to address the  case that $\d\neq\alpha_{m+1}$. Let $\d=\{\d_1,\cdots,\d_l\}$. Since by definition
\begin{align*}
D^\d v_{N+1}(t,x)=D_{\d_1}^{|\d_1|}v_{N+1}(t,x)\cdots D_{\d_l}^{|\d_l|}v_{N+1}(t,x),
\end{align*}
then
\begin{align*}
\E\|D^{|\d|} v_{N+1}(t,x)&\|^p_{H^{\otimes m+1}}\\&=\E(\|D^{|\d_1|}v_{N+1}(t,x)\|^p_{H^{\otimes|\d_1|}}\cdots \|D^{|\d_l|}v_{N+1}(t,x)\|^p_{H^{\otimes|\d_l|}}),
\end{align*}
where, $l\geq 2$, and $|\d_1|+\cdots+|\d_l|=m+1.$ 
Then by the generalized H\"older's inequality,
\begin{align*}
\left\{\E\|D^{|\d|} v_{N+1}(t,x)\|^p_{H^{\otimes m+1}}\right\}^{l}\leq\E\|D^{|\d_1|}v_{N+1}(t,x)\|^{lp}_{H^{\otimes|\d_1|}}&\times\cdots\\
\times \E\|D^{|\d_l|}&v_{N+1}(t,x)\|^{lp}_{H^{\otimes|\d_l|}}.
\end{align*}
Equivalently,
\begin{align*}
e^\frac{-\beta t}{lp}&\left\{\E\|D^{|\d|} v_{N+1}(t,x)\|^p_{H^{\otimes m+1}}\right\}^{1/p}\\
&\leq \left\{e^{-\beta t}\E|\G_{t,x}^{|\d_1|}v_{N+1}|^{lp}\right\}^\frac{1}{lp}\times\cdots
\times \left\{e^{-\beta t}\E|\G_{t,x}^{|\d_l|}v_{N+1}|^{lp}\right\}^\frac{1}{lp}.
\end{align*}
We optimize, first the right-hand-side and then the left-hand-side of the latter inequality over all $t>0$ and $x\in\T$ in order to find that 
\begin{align*}
\|\G^{\d} v_{N+1}\|_{\frac{\beta}{l},p}\leq \|\G^{|\d_1|}v_{N+1}\|_{\beta,lp}\cdots \|\G^{|\d_l|}v_{N+1}\|_{\beta,lp}.
\end{align*}
If we replace $\beta$ by $l\beta$, then we have 
\begin{align}\label{d12}
\|\G^{\d} v_{N+1}\|_{\beta,p}\leq \|\G^{|\d_1|}v_{N+1}\|_{l\beta,lp}\cdots \|\G^{|\d_l|}v_{N+1}\|_{l\beta,lp}\notag\\
=\Pi_{j=1}^l\|\G^{|\d_j|}v_{N+1}\|_{l\beta,lp}<\infty.
\end{align}
Therefore, $v_n\in D^{m+1,p}$ for all $n$. This finishes the proof. 
\end{proof}
\begin{rmk}
Similar to what we did in Proposition~\ref{D1pu}, we can iterate ~\eqref{recursive}, and choose $\beta>0$ sufficiently large to obtain 
\begin{equation*}
\sup_n\E\|D^mv_{n}(t,x)\|^p_{H^{\otimes m}}<\infty,
\end{equation*}
which in turn implies that 
\begin{equation}\label{Dmpvn}
\sup_n\|v_n(t,x)\|_{m,p}<\infty.
\end{equation}
\end{rmk}
\subsection{Proof of the Theorem~\ref{4.1}}
We prove Theorem~\ref{4.1}  by applying induction on the order of the derivative $k$. In Propostion~\ref{D1pu} we showed that $u\in \mathbf D^{1,p}$ and its derivative $Du$ satisfies~\eqref{Du} for $k=1$.  

Assume now that $u\in \mathbf D^{k,p}<\infty$ for all $k\leq m-1$, $p\geq1$ and the $k$th derivative $D^ku$ satisfies~\eqref{Du} for $k=1,\cdots,m-1$. This together with~\eqref{Dmpvn} imply that $u(t,x)\in \mathbf D^{m,p}$. 
Next we show that ~\eqref{Du} also holds for $k=m$. This proof is basically repeating what we did for the proof of ~\eqref{recursive}, and therefore we avoid going through the details. Define 
\begin{eqnarray*}
c_n^2(t,x)=
\frac{1}{2}\|\G_{t,x}^m\left(v_{n+1}-u\right)\|_{L^p(\Omega)}^2,\hspace{1.1 in}\\
b^2_n(t,x)=\left\{\E\left\|q_{t-\bullet}(x,\ast)D^{m-1}_\diamond(\sigma(v_n(\bullet,\ast))-\sigma(u(\bullet,\ast)))\right\|_{H^{\otimes m}}^p\right\}^{2/p}\\
a^2_n(t,x)=\left\{\E\left\|\intt\intT q_{t-r}(x,z)D^m(\sigma(v_n(r,z))-\sigma(u(r,z)))w(dr,dz)\right\|_{H^{\otimes m}}^p\right\}^{2/p}.
\end{eqnarray*}
The goal is to show that $\lim_{n\to\infty}\sup_{0<t\leq T}\sup_{x\in\T}c_n^2(t,x)=0$. 
By the triangular inequality,
\begin{eqnarray*}
c_n^2(t,x)\leq b^2_n(t,x)+a^2_n(t,x),
\end{eqnarray*}
A similar argument as the proof of Proposition~\ref{Gv_n} leads to the following bound on $b_n$:
\begin{eqnarray}
e^{-2\beta t/p}b_n(t,x)^2\leq m\|\G^{m-1}(\sigma(v_n)-\sigma(u))\|_{\beta,p}^2\Upsilon(2\beta/p)\notag.
\end{eqnarray}
Finding an upper bound for $a_n(t,x)$ is similar to what we have done for $F_1$, which led to~\eqref{F_1betap}. For $a_n$ we have 
\begin{eqnarray*}
e^{-2\beta t/p}a_n^2(t,x)
\leq C_p\|\G^m(\sigma(v_n)-\sigma(u))\|_{\beta,p}^2\Upsilon(2\beta/p).
\end{eqnarray*}
Another application of~\eqref{Gssigma}, together with the induction hypothesis shows that 
$$e^{-2\beta t/p}a_n^2(t,x)\leq C_p\left(\lambda_n+\|\G^m(v_n-u)\|_{\beta,p}^2\right)\Upsilon(2\beta/p),$$
where $\lambda_n$ is independent of $t$ and $x$, and $\lambda_n\to0$ as $n\to\infty$. Therefore 
\begin{align*}
e^{-2\beta t/p}c_n^2(t,x)\leq
K_{p,m}(\theta_n+\|\G^m(v_n-u)\|_{\beta,p}^2)\Upsilon(2\beta/p),
\end{align*}
where $K_{p,m}=\max\{C_p,m\}$ and $\theta_n$ is independent of $t$ and $x$ and $\theta_n\to0$ as $n\to\infty$.
Therefore, by choosing $\beta$ sufficiently large so that $K_{p,m}\Upsilon(2\beta/p)<1$, we have 
\begin{align*}
\|\G^m(v_{n+1}-u)\|_{\beta,p}^2\leq C_{k,m,\beta}(\theta_n+\|\G^m(v_n-u)\|_{\beta,p}^2).
\end{align*}
The latter inequality implies that $\|\G^m(v_n-u)\|_{\beta,p}^2\to0$ as $n\to\infty$ which is equivalent to want we wanted to prove.
\begin{rmk}
The value of $\beta$ transfers through the induction steps; i.e., its value in the $m$th step must be at least as large as its value in $(m-1)$th step. 
\end{rmk}
\section{Analysis of the Malliavin Matrix}
Next we study the $L^p(\Omega)$-integrability of the inverse of the Malliavin matrix. Here is the first place where we use the second assumption, Hypothesis $\bf H\-$\ref{condition2}, of this paper,  which asserts that there are $1<\alpha<\beta\leq2$ and $0<C_1<C_2,$ such that 
\begin{equation*}
C_1|n|^{\alpha}\leq\Re\varPhi(n)\leq C_2|n|^{\beta}.
\end{equation*} 
When ~\eqref{A SPDE} is linear, we know that if $\beta\leq1$, then a solution does not exist. The case that $\Phi(n)=n^2$ is well known \cite{BP98}. In this section, we want to show that for every $(t,x)\in E_T$, and $p\geq 2$,
\begin{equation}\label{I}
\E(\|Du(t,x)\|^{-p})<\infty.
\end{equation}

\begin{lma}\label{V(t)}
Let $u$ be the solution to SPDE~\eqref{A SPDE}. Let $p\geq 1$.
\begin{enumerate}
\item \label{V(T)1} If we define 
$$V(t)=\sup_{x\in[0,2\pi]}\E\left(\int_0^t\intT |D_{s,y}u(t,x)|^2dy\-ds\right)^{p/2},$$
then 
\begin{eqnarray*}
V(t)\leq C_{T,p}t^{(\alpha-1)p/2\alpha}.
\end{eqnarray*}
\item\label{V(T)2}  If we fix $t$ and for any $\delta\in(0,t)$ define 
$$W(\delta)=\sup_{x\in[0,2\pi]}\E\left(\int_{t-\delta}^t\intT |D_{s,y}u(t,x)|^2dy\-ds\right)^{p/2},$$
then 
\begin{eqnarray*}
W(\delta)\leq C_{T,p}\delta^{(\alpha-1)p/2\alpha}.
\end{eqnarray*}
\end{enumerate}

\begin{proof}
We prove only the the first part, as the second part can be proved similarly. Because 
\begin{eqnarray*}
D_{sy}u(t,x)=q_{t-s}(x,y)\sigma(u(s,y))+\int_0^t\intT q_{t-r}(x,z)D_{s,y}\sigma(u(r,z))\-w(dr,dz),
\end{eqnarray*}
it follows that
\begin{align*}
\|D&u(t,x)\|_H\leq\lip_\sigma\left\{\intt\intT q_{t-s}^2(x,y)\-dy\right\}^{1/2}\\
&+\lip_{\sigma'}\left\{\intt\intT\left|\int_0^t\intT q_{t-r}(x,z)D_{s,y}u(r,z)\-w(dr,dz)\right|^2\-dy\-ds\right\}^{1/2}.
\end{align*}
If we let $C_\sigma=\max\{\lip_\sigma,\lip_{\sigma'}\}$, then by ~\eqref{LB},
\begin{align*}
\frac{1}{C_\sigma}&\|Du(t,x)\|_H\leq C_\alpha t^\frac{\alpha-1}{2\alpha}\\
&+\left\{\intt\intT\left|\int_0^t\intT q_{t-r}(x,z)D_{s,y}u(r,z)\-w(dr,dz)\right|^2\-dy\-ds\right\}^{1/2}.
\end{align*}
If we let $C=\left(2C_\sigma\max\{C_\alpha, 1\}\right)^p$, then take the expectation of the $p$th power to get
\begin{eqnarray*}
\frac{1}{C}\E\|Du(t,x)\|^p_H\leq t^\frac{(\alpha-1)p}{2\alpha}+\E\left\|\int_0^t\intT q_{t-r}(x,z)D(u(r,z))\-w(dr,dz)\right\|^p_H,
\end{eqnarray*}
where $\|\cdot\|_H$ denotes the Hilbert space norm with respect to variables $s$ and $y$.
Then by Burkholder's inequality for the Hilber-space-valued martingales we have 
\begin{eqnarray*}
\frac{1}{C_{p,\sigma}}\E\|Du(t,x)\|^p_H\leq t^\frac{(\alpha-1)p}{2\alpha}+\E\left(\int_0^t\intT q^2_{t-r}(x,z)\|D(u(r,z))\|^2_H\-dz\-dr\right)^{p/2},
\end{eqnarray*}
where the constant $C_{p,\sigma}$ depends on $p$ through the Burkholder's inequality. 
Next, by observing
\begin{equation*}
q^2_{t-r}(x,z)\|D(u(r,z))\|^2_H=
q^\frac{2p-4}{p}_{t-r}(x,z)\left(q_{t-r}^\frac{4}{p}(x,z)\|D(u(r,z))\|^2_H\right),
\end{equation*}
we may apply the H\"older inequality to obtain
\begin{align*}
\int_0^t\intT q^2_{t-r}(x,z)\|D(u(r,z))\|^2_H\-dz\-dr\leq\left(\int_0^t\intT q^2_{t-r}(x,z)drdz\right)^{(p-2)/p}\times\\
\left(\int_0^t\intT q_{t-r}^2(x,z)\|D(u(r,z))\|^p_H\-dz\-dr\right)^{p/2}.
\end{align*}
Another application of~\eqref{LB} yields 
\begin{align*}
\int_0^t\intT q^2_{t-r}(x,&z)\|D(u(r,z))\|^2_H\-dz\-dr\\
&\leq C_{p,\alpha} t^\frac{(p-2)(\alpha-1)}{p\alpha}\left(\int_0^t\intT q_{t-r}^2(x,z)\|D(u(r,z))\|^p_H\-dz\-dr\right)^{p/2}.
\end{align*}
Then, for a new constant $C$, we have
\begin{align*}
\frac{1}{C}\E\|&Du(t,x)\|^p_H\\
&\leq t^\frac{(\alpha-1)p}{2\alpha}+C_{p,\alpha} t^\frac{(p-2)(\alpha-1)}{2\alpha}\int_0^t\intT q_{t-r}^2(x,z)\E\|D(u(r,z))\|^p_H\-dz\-dr\\
&\leq t^\frac{(\alpha-1)p}{2\alpha}+ t^\frac{(p-2)(\alpha-1)}{2\alpha}\int_0^t\sup_{z\in[0,2\pi]}\E\|D(u(r,z))\|^p_H(t-r)^{-1/\alpha}\-dr.
\end{align*}
Then
\begin{eqnarray*}
V(t)\leq C\left(t^\frac{(\alpha-1)p}{2\alpha}+ t^\frac{(p-2)(\alpha-1)}{2\alpha}\int_0^tV(r)(t-r)^{-1/\alpha}\-dr\right).
\end{eqnarray*}
Next, apply H\"older's inequality to the integral on the right in order to find that 
\begin{eqnarray*}
\int_0^tV(r)(t-r)^{-1/\alpha}dr\leq\left(\int_0^tV(r)^{p_1}dr\right)^{1/p_1}\left(\int_0^T(t-r)^{-q_1/\alpha}\right)^{1/q_1},
\end{eqnarray*}
where $p_1=(\alpha+1)/2$ and $q_1=(\alpha+1)/(\alpha-1)$. Because $q_1/\alpha<1$,
$$\left(\int_0^T(t-r)^\frac{-q_1}{\alpha}\right)^\frac{1}{q_1}<\infty.$$
Therefore there is $C$ such that
$$V(t)\leq C\left(t^\frac{(\alpha-1)p}{2\alpha}+ t^\frac{(p-2)(\alpha-1)}{2\alpha}\left(\int_0^tV(r)^{p_1}dr\right)^{1/p_1}\right).$$
Consequently for some $C>0,$
$$V(t)^{p_1}\leq C\left(t^\frac{(\alpha-1)pp_1}{2\alpha}+ t^\frac{p_1(p-2)(\alpha-1)}{2\alpha}\int_0^tV(r)^{p_1}dr\right).$$
Again, since $0\leq t\leq T$, then we can choose $C$ such that 
$$V(t)^{p_1}\leq C\left(t^\frac{(\alpha-1)pp_1}{2\alpha}+\int_0^tV(r)^{p_1}dr\right).$$
Then by the Gronwall's lemma we have 
$$V(t)\leq Ct^\frac{(\alpha-1)p}{2\alpha}.$$
This concludes the proof of \eqref{V(T)1}. \eqref{V(T)2} is proved similarly.
\end{proof}
\end{lma}
The following corollary is an estimate on the Malliavian's derivative of the solution of the equation~\eqref{A SPDE}.
\begin{crl}\label{I^p}
Let $u$ be the solution to the equation~\eqref{A SPDE}, and $\varPhi$ the L\'evy exponent corresponding the differential operator $\sL$. Define 
\begin{equation}\label{I_del}
I_\delta=\int_{t-\delta}^t\intT\left|\int_0^t\intT q_{t-r}(x,z)D_{sy}\sigma(u(r,z)\-w(dr,dz)\right|^2dy\-ds,
\end{equation}
where $q=q_t(x)$ is the transition density corresponding to $\sL$.
If the $\varPhi$ satisfies Hypothesis {\bf H\-\ref{condition2}}, then 
$$\E\left(|I_\delta|^p\right)\leq C\delta^{2p(\alpha-1)/\alpha},$$
\begin{proof}
By the Burkholder's inequality
\begin{eqnarray*}
\E(|I_\delta|^p)=\E\left(\left|\int_{t-\delta}^t\intT\left|\int_s^t\intT q_{t-r}(x,z)D_{s,y}\sigma(u(r,z))\-w(dr,dz)\right|^2\-dy\-ds\right|^p\right)\\
\leq c_p\lip_\sigma^{2p}\E\left|\int_{t-\delta}^t\intT q^2_{t-r}(x,z)\left(\int_{t-\delta}^r\intT|D_{sy}u(r,z)|^2dyds\right)dzdr\right|^p.
\end{eqnarray*}

Raising to the power $1/p$ and applying Minkowski's inequality gives us
\begin{eqnarray*}
\left\{\E(|I_\delta|^p)\right\}^{1/p}
\leq c^{1/p}_p\lip_\sigma^{2}\left\{\E\left|\int_{t-\delta}^t\intT q^2_{t-r}(x,z)\left(\int_{t-\delta}^r\intT|D_{s,y}u(r,z))|^2dyds\right)dzdr\right|^p\right\}^{1/p}\\
\leq c^{1/p}_p\lip_\sigma^{2}\left(\int_{t-\delta}^t\intT q^2_r(z,x)dzdr\right)\sup_{(r,z)\in[0,\delta]\times[0,2\pi]}\left\{\E\left|\int_{t-\delta}^t\intT|D_{s,y}u(r,z))|^2dyds\right|^p\right\}^{1/p}.
\end{eqnarray*}
Then, by ~\eqref{LB}, and Lemma~\ref{V(t)},
\begin{align*}
\E(|I_\delta|^p)
&\leq C\left(\int_0^\delta\intT q^2_r(z,x)dzdr\right)^p\\&
\times\sup_{(r,z)\in[0,\delta]\times[0,2\pi]}\E\left|\int_{t-\delta}^t\intT|D_{s,y}u(r,z))|^2dyds\right|^p\\
&\leq C\delta^{(\alpha-1)p/\alpha}\delta^{(\alpha-1)p/\alpha}.
\end{align*}

\end{proof}
\end{crl}
Finally we quote from \cite[page 97]{minicourse} a lemma which allows us to put together the results of Lemma~\ref{V(t)} and Corollary~\ref{I^p} and prove the existence of the negative moments~\eqref{I}. 

\begin{lma}\label{5.3} Let $F$ be nonnegative random variable. Then property~\eqref{I} holds for all $p\geq 2$ if and only if for every $q\in[2,\infty)$ there exists $\ve_0=\ve_0(q)>0$, such that 
$$\P(\|Du(t,x)\|^2_H<\ve)<C\ve^q,$$
for all $\ve<\ve_0$.
\end{lma}
\begin{proof}[Proof of Theorem~\ref{main}]
We need only to show that ~\eqref{I} holds for every $(t,x)\in E_T$ and all $p\geq 2$.
Let $q=q_t(x)$ be the transition density corresponding to $\varPhi$ and $\sL$.
Since
\begin{eqnarray*}
q_{t-s}(y,x)\sigma(u(s,y))=D_{sy}u(t,x)-\int_0^t\intT q_{t-r}(x,z)D_{sy}\sigma(u(r,z))\-w(dr,dz),
\end{eqnarray*}

considering the fact that $\sigma\geq c>0$, then
\begin{eqnarray*}
|D_{sy}u(t,x)|^2\geq\frac{c^2}{2}q^2_{t-s}(y,x)-\left|\int_0^t\intT q_{t-r}(x,z)D_{sy}\sigma(u(r,z))\-w(dr,dz)\right|^2.
\end{eqnarray*}
Therefore
\begin{eqnarray*}
\|Du(t,x)\|_H^2\geq\int_{t-\delta}^t\intT|D_{sy}u(t,x)|^2dy\-ds\hspace{3 in}\\
\geq\frac{c^2}{2}\int_{t-\delta}^t\intT q^2_{t-s}(y,x)dyds
-\int_{t-\delta}^t\intT\left|\int_0^t\intT q_{t-r}(x,z)D_{sy}\sigma(u(r,z))\-w(dr,dz)\right|^2dyds.
\end{eqnarray*}
If we let $\tau=t-s$ in the first integral on the right, then we get 
\begin{eqnarray*}
\|Du(t,x)\|_H^2\geq
J_\delta-I_\delta\hspace{.5 in}\forall \delta\in(0,t),
\end{eqnarray*} 
where $I_\delta$ is defined in~\eqref{I_del} and 
$$J_\delta:=\frac{c^2}{2}\int_0^\delta\|q_u\|_{L^2(\T)}du.$$
If we choose $\delta>0$ such that $J_\delta-\ve>0$, then by the Chebyshev's inequality 
\begin{equation*}
\P(\|D_{sy}u(t,x)\|_H^2<\ve)\leq\P(J_\delta-I_\delta<\ve)\leq\frac{\E|I_\delta|^p}{(J_\delta-\ve)^p}.
\end{equation*}
Then, by Corollary~\ref{I^p} and ~\eqref{LB} we have
$$\P(\|Du(t,x)\|_H^2<\ve)\leq\frac{C_1\delta^{2(\alpha-1)p/\alpha}}{(\frac{C}{2}\delta^{(\beta-1)/\beta}-\ve)^p}.$$
Take $\delta=(4\epsilon/C)^\frac{\beta}{\beta-1}$ to get 
$$\P(\|Du(t,x)\|_H^2<\epsilon)\leq C_{\alpha,p}\ve^{\theta p},$$
where $\theta=\frac{2\beta(\alpha-1)}{\alpha(\beta -1)}-1>0$. This and Lemma~\ref{5.3} complete the proof.\end{proof}
\bibliography{growthrefs2}
\bibliographystyle{plain}

\end{document}